\documentclass[11pt]{article}

\usepackage[T1]{fontenc}
\usepackage[utf8]{inputenc}
\usepackage{lmodern}
\usepackage[margin=1in]{geometry}
\usepackage{amsmath,amssymb,amsthm,mathtools}
\usepackage{microtype}
\usepackage{xcolor}
\usepackage{enumitem}
\usepackage{booktabs}
\usepackage{array}
\usepackage{hyperref}

\hypersetup{
  colorlinks=true,
  linkcolor=blue!55!black,
  citecolor=blue!55!black,
  urlcolor=blue!55!black,
  pdftitle={Polynomial Bohnenblust--Hille Bounds on Finite Cyclic Products},
  pdfauthor={Research draft}
}

\allowdisplaybreaks
\setlist[itemize]{leftmargin=2em,itemsep=0.25em,topsep=0.35em}
\setlist[enumerate]{leftmargin=2em,itemsep=0.25em,topsep=0.35em}

\newtheorem{theorem}{Theorem}[section]
\newtheorem{proposition}[theorem]{Proposition}
\newtheorem{lemma}[theorem]{Lemma}
\newtheorem{corollary}[theorem]{Corollary}
\theoremstyle{definition}

\theoremstyle{remark}
\newtheorem{remark}[theorem]{Remark}

\newcommand{\E}{\mathbb{E}}
\newcommand{\C}{\mathbb{C}}
\newcommand{\N}{\mathbb{N}}

\newcommand{\Z}{\mathbb{Z}}
\newcommand{\wh}{\widehat}
\newcommand{\norm}[1]{\left\lVert #1\right\rVert}
\newcommand{\abs}[1]{\left\lvert #1\right\rvert}
\newcommand{\e}{\mathrm{e}}
\newcommand{\supp}{\operatorname{supp}}
\newcommand{\BH}{\mathrm{BH}}
\DeclareMathOperator{\Bin}{Bin}

\title{\textbf{Polynomial Bohnenblust--Hille Bounds\\ for product of cyclic groups}
\\[0.35em]
\large A Potts-hypercontractive weighted bootstrap}
\date{September 2, 2026}
\author{Joseph Slote, Alexander Volberg}
\date{September 3, 2026}

\begin{document}
\maketitle


\begin{abstract}
Fix an integer $K\ge2$, and let $C_K^n =\{(e^{\frac{2\pi ij}{K}})_{j=0}^{K-1}\}^n$ be the product of
cyclic groups of order $K$.  For a Fourier character $\chi_\alpha$, let
$s(\alpha)$ be the number of active coordinates.  We give a self-contained
proposed proof that the dimension-free Bohnenblust--Hille constants governed
by interaction order grow polynomially: if
$p_d=2d/(d+1)$ and
\[
 \gamma_2=\frac12,
 \qquad
 \gamma_K=\frac{K\log(K-1)}{4(K-2)}\quad(K\ge3),
\]
then the $\ell^{p_d}$ norm of Fourier coefficients $\{\hat f(\alpha)\}$ is bounded  by $L^\infty$ norm of $f$ multiplied by $C(K) d^{4\gamma_K+5}$.
The constant $C(K)$ is actually at most of the order $K^{5/2}$.
\end{abstract}

\tableofcontents

\section{Introduction}

We were made aware of the paper \cite{I} on September 7. It was published online on August 27. For Hamming cube case $K=2$ it has the same result with a slightly worse exponent.

\medskip

Bohnenblust--Hille inequality has an old and venerable history. It was first proved in \cite{BohnenblustHille} in 1931 and gave the 
$n$-free estimate of a certain coefficient norm of an analytic  polynomial  of variables $n$ via its maximum on multi-torus $\mathbb T^n$.
The constant was independent of $n$ but dependent in a superexponential fashion on the degree $m$ of the polynomial.
The result was needed to answer a question of Harold Bohr concerning a convergence property of Dirichlet series.

Much later this result was revised and the constant was improved in \cite{DefantFrerickOrtegaCerdaOunaiesSeip}, and the constant became exponential in degreee $m$.

A bit later \cite{BayartPellegrinoSeoane-Sep\'ulveda} proved a subexponential in $m$ estimate for multi-torus case.

Very soon \cite{DefantMastyloPerez} proved a similar subexponential estimate for polynomials on Hamming cube.

This latter result was used in a learning theory by Eskenazis--Ivanisvili \cite{EskenazisIvanisvili} 
to obtain an optimal number of queries to PAC (probably approximately correct) learning of function of low degree on Boolean cube.

In note \cite{SloteVolberg} we gave a polynomial in degree estimate for the constant of the Bohnenblast--Hille inequality on Boolean cube.

The argument in \cite{SloteVolberg} adapted the weighted graded square-function method introduced in
Pellegrino and Teixeira \cite{PellegrinoTeixeira}  Section 5. There it was introduced  for analytic polynomials.  
There are several new and exciting ideas in \cite{PellegrinoTeixeira}. One of them is a bicoloring of variables, another is a special square function
on coefficients of polynomials. 

On the Boolean cube,
Bonami--Beckner hypercontractivity supplies simultaneous row and column
control after a random bicoloring of the coordinates. Bicoloring idea comes from \cite{PellegrinoTeixeira}, but in the case of Hamming cube it can be realized in a simpler way.
Every monomial on Hamming cube is 
square-free, so the dominant-power compression regime from the analytic
setting disappears.  


The current note is devoted to Bohnenblust--Hille inequality on cyclic group. Now we have $K\ge 3$ and the group $C_K$ of $K$-th
roots of unity. 
Polynomial $F$ of $n$ variables on $C_K^n$ has total degree $m$ but the degree in each variable is at most $K-1$.

Such $n$-free Bohnenblust--Hille inequality was proved in \cite{BeckerKleinSloteVolbergZhang}, \cite{SloteVolbergZhang}, 
but the constant was exponential in the degree $m$ of the polynomial (and of course also depending on $K$).

Here the polynomial in $m$ estimate is obtained. The method repeats this of \cite{PellegrinoTeixeira} Section 5 by exploring the excited ideas introduced in this paper.

Also the proof follows closely the one in \cite{SloteVolberg}, but with some caveats. For example, 
in the cyclic group situation we know \cite{SloteVolbergZhang}, \cite{DefantEtAlHamming} that the spectral projection on the linear combination of monomials
$z^\alpha$ such that the cardinality  of $\supp\alpha$  has an estimate independent on dimension $n$. 

On Hamming cube ($K=2$) such spectral projection is the same as the projection on the fixed homogeneity part, but for $K\ge 3$
these are two different things.

\medskip

Bohnenblust--Hille inequality on cyclic groups was used to get a quantum Bohnenblust--Hille inequality \cite{SloteVolbergZhangNC}, \cite{SloteVolbergZhang},
Interestingly the constant is exponential in $m$ because of the reduction of non-commutative case to the commutative cyclic group case. 
But what is even more interesting is that this exponential nature of the non-commutative Bohnenblust--Hille inequality cannot be made sub-exponential, see \cite{Slote}.

\medskip

The architecture of the proof repeats this of \cite{PellegrinoTeixeira} Section 5 to a large extent. It is the weighted graded bootstrap of
Pellegrino--Teixeira: a weighted square function, simultaneous row and column
control, 
and a strict binary-entropy contraction for
balanced two-block degrees.  

Its implementation on support layers follows the
Boolean-cube  argument \cite{SloteVolberg}: whole-coordinate random bicoloring, exact
reconstruction of the parent Fourier index, binomial central capture, 
and a separate low-level estimate.  

The support mode splitting is much more suitable here than the homogeneous degree mode splitting. This phenomena 
was already noticed in \cite{SloteVolbergZhang} and in \cite{DefantEtAlHamming}, but  \cite{DefantEtAlHamming}, gives a much better estimate
on the spectral projection $L^\infty$ norm that the one in \cite{SloteVolbergZhang}. This better estimate is crucial, it is used in Lemma \ref{lem:projection} below.

\medskip

The only new
high-degree analytic input taken from the cyclic-group paper of
Pellegrino--Raposo Junior  \cite{PellegrinoRaposo} is its Potts hypercontractive radius
$\rho\le(p-1)^{\gamma_K}$.  Three fixed low interaction levels are closed
using the previously known support-sensitive cyclic Bohnenblust--Hille
inequality and an interior Bernstein estimate.  An explicit contraction
constant $0.820773$ important for bootstrapping argument that finishes the proof. It   is obtained uniformly in $K$.

\

\subsection{The problem and the proposed result}

Let
\[
 C_K=\{1,\omega,\ldots,\omega^{K-1}\},
 \qquad \omega=\e^{2\pi i/K},
\]
with normalized counting measure.  We identify $C_K$ with $\Z/K\Z$ and work
multiplicatively because this makes the Fourier characters monomials.
For $\alpha\in\{0,\ldots,K-1\}^N$, set
\[
 \chi_\alpha(x)=x^\alpha=\prod_{j=1}^N x_j^{\alpha_j},
 \qquad
 s(\alpha)=\#\{j:\alpha_j\ne0\}.
\]
The quantity $s(\alpha)$ is the \emph{interaction order}.  It records the
number of coordinates on which the character depends, independently of the
local nonzero degrees $\alpha_j\in\{1,\ldots,K-1\}$.

Pellegrino and Raposo Junior recently proved a subexponential bound for the
corresponding dimension-free constants and, in particular, for the larger
interaction-order class \cite{PellegrinoRaposo}.  The purpose of this note is
to show that the weighted entropy mechanism of Pellegrino and Teixeira
\cite{PellegrinoTeixeira}, when combined with Potts hypercontractivity, appears
to improve the degree dependence from subexponential to polynomial.

Our main statement is the following.

\begin{theorem}[Proposed polynomial cyclic Bohnenblust--Hille bound]
\label{thm:main-intro}
For every fixed integer $K\ge2$, there is a constant $C_K<\infty$ such that,
for every $d,N\ge1$ and every complex-valued function
$f:C_K^N\to\C$ whose Fourier transform is supported on
$s(\alpha)\le d$,
\begin{equation}
\label{eq:int-main-intro}
 \left(
  \sum_{s(\alpha)\le d}
  \abs{\wh f(\alpha)}^{p_d}
 \right)^{1/p_d}
 \le C_K d^{4\gamma_K+5}\norm f_\infty,
 \qquad
 p_d=\frac{2d}{d+1}.
\end{equation}
Here
\begin{equation}
\label{eq:gamma-intro}
 \gamma_2=\frac12,
 \qquad
 \gamma_K=\frac{K\log(K-1)}{4(K-2)}\quad(K\ge3).
\end{equation}
Consequently, for $K\ge3$ the displayed power is
\begin{equation}
\label{eq:power-explicit}
 4\gamma_K+5
 =5+\frac{K\log(K-1)}{K-2}.
\end{equation}
The same bound holds for functions of canonical total degree at most $d$.
\end{theorem}

The result is for fixed $K$.  The multiplicative constant may depend on $K$,
and the exponent itself grows like $5+\log K$ as $K\to\infty$. The constant $\gamma_K$ is exactly half the reciprocal of the classical complete-graph of $K$ vertices log-Sobolev constant, see \cite{DiaconisSaloffCoste}.

\subsection{Provenance of the ingredients}
\label{subsec:provenance}

Because the proof combines several recent ideas, it is useful to separate
precisely what comes from where.

\paragraph{The Pellegrino--Teixeira architecture.}
The following high-level devices are taken from the weighted stage of
\cite{PellegrinoTeixeira}:
\begin{itemize}
\item the weighted graded square function and its bootstrap constant;
\item simultaneous row and column estimates for all bidegrees;
\item the fractional two-block coefficient interpolation;
\item the variable-angle mixed summation lemma;
\item the conversion of balanced bidegrees into a strict binary-entropy
contraction, followed by bootstrapping argument
\end{itemize}
The Pellegrino--Teixeira paper also supplies the general ancestry philosophy:
a two-block descendant must retain enough information to recover its parent
coefficient before absolute values are taken.  Its actual analytic splitting
scheme is not copied verbatim here.  All estimates needed in the present
support-layer implementation are reproduced below.  We do not use the
Pellegrino--Teixeira dominant-power compression, 
an active coordinate contributes
one support unit, regardless of its local frequency degree.

\paragraph{The Boolean-cube adaptation.}
The passage from analytic homogeneous degree to support layers grading follows the
Boolean companion argument rather than the literal analytic split:
\begin{itemize}
\item random bicoloring is performed on whole active coordinates, not on units
of algebraic exponent;
\item the pair of colored restrictions retains every local degree and
uniquely recovers the parent Fourier index;
\item the number of active coordinates sent to either block is binomial, which
gives a uniform central-capture probability;
\item a finite ambient-dimension cutoff is introduced before bootstrapping, so
finiteness is not assumed circularly;
\item the few low support levels are extracted by a one-variable noise
polynomial and an interior Bernstein estimate.
\end{itemize}
The cyclic case has an asymmetric negative-noise interval, but this interval
still contains zero in its interior.  This is the only modification needed in
the low-level projection argument.

\paragraph{The cyclic-group paper.}
The only genuinely new \emph{high-degree analytic} input imported from
Pellegrino--Raposo Junior \cite{PellegrinoRaposo} is Potts hypercontractivity \cite{DiaconisSaloffCoste}, \cite{Gross}:
\[
 \norm{T_\rho^{\otimes N}g}_2\le\norm g_p
 \quad\text{for}\quad
 0\le\rho\le(p-1)^{\gamma_K}.
\]
Their interaction-order notation and Fourier normalization are also adopted.
The present proof does \emph{not} use their homogenization, symmetric
multilinear form, orbit normalization, mixed polarization, block recurrence,
or asymptotic optimization.

\paragraph{Known information on $L^2\to L^p$ hypercontractivity of the specific noise operator on cyclic group $C_K$.}
There are several different noise operators on functions on $C_K^n$. 
One is $T_\rho^{\otimes n}$, where
$$
T_\rho \chi_j = \rho^{\psi(j)} \chi_j, \quad \psi(j) =\min (j, K-j)\,.
$$
another one is $T_\rho^{\otimes n}$ with just
$$
T_\rho =\rho I +(1-\rho) \mathbb E\,.
$$
The latter one is responsible for the so-called Potts hypercontractivity and its hypercontractive properties are studied in \cite{DiaconisSaloffCoste}, \cite{Gross}.
It is known that $(2,p)$ hypercontractivity happens with 
$$
\rho=(p-1)^{\gamma_K}.
$$

Even though we do not need the first one we list below some known results with $\rho$ that gives hypercontractivity.

The short historical answer is:

For $K=3$, the sharp value is different and comes from the Latała-Oleszkiewicz biased two-point inequality via Wolff's reduction.

$K=4$: Beckner-Janson-Jerison (1983), $(p-1)^{1/2}$ sharp, same as Boolean.

$K=3$: Wolff + Latała-Oleszkiewicz, sharp $L^2\to L^p$, different constant.

General $K$: Junge-Palazuelos-Parcet-Perrin (2017), many optimal $L^q \to L^2$ results for even $q$.

2025–2026 preprints (Yao; Xie–Zhang): claim the $L^q\to L^p$ full sharp constant $\rho=\sqrt{\frac{p-1}{q-1}}, K\ge 4$.

This is actually much closer to the Boolean world than one would have expected twenty years ago.

\paragraph{A small additional base-case input.}
There is one caveat to the phrase "using Potts hypercontractivity". To close the levels
$r=1,2,3$, we also use the already-known dimension-free support-order
Bohnenblust--Hille inequality, only in those three fixed degrees.  This can be
supplied independently by the support-sensitive Hamming-scheme theorem of
Defant--Galicer--Mansilla--Masty\l o--Muro \cite{DefantEtAlHamming}.  No
asymptotic information about its constants is used.  Thus Potts
hypercontractivity is the only cyclic-specific input in the high-degree
bootstrap.

\subsection{Important points in the proof}
\label{subsec:audit}

Before presenting the proof, we record the points at which the argument is
most vulnerable and how they are resolved.

\begin{enumerate}
\item \textbf{Circular finiteness.}
The weighted constant is first defined with an ambient cutoff $N$.  It is
finite by finite dimensionality.  Only after bootstrapping do we obtain a bound
independent of $N$.

\item \textbf{Hypercontractive radius $\rho$.}
The exact radius used below is
$\rho=(p-1)^{\gamma_K}$, with $\gamma_K$ as in
\eqref{eq:gamma-intro}; this is Proposition 3.1 of
\cite{PellegrinoRaposo}.  

\item \textbf{Complex-valued functions.}
The Potts semigroup is positive.  Hence
$\abs{T_\rho g}\le T_\rho\abs g$, so the real hypercontractive estimate
extends with the same constant to complex functions.

\item \textbf{Failure at one-coordinate bidegrees.}
The radius associated with $p_1=1$ vanishes.  We therefore begin the central
bootstrap at total interaction order $4$ and explicitly require both block
orders to be at least $2$.

\item \textbf{Modular collisions.}
A local degree of one variable is never split arithmetically modulo $K$.  The entire
labeled coordinate is sent to one block.  The two restrictions have disjoint
coordinate supports and their union recovers the original frequency vector,
so distinct parent coefficients cannot merge.

\item \textbf{Noncontractive interaction projections.}
The projections $f\mapsto f_r$ need not be $L^\infty$ contractions.  They are
used only for $r\le3$, where the extended Potts kernel is Markov on
$[-1/(K-1),1]$ and an interior Bernstein estimate gives a polynomial cost
$O_K(M^r)$.

\item \textbf{Loss of entropy through the Potts radius.}
Removing the Potts weights costs at most $3^{2\gamma_K}$, uniformly in the
bidegrees.  The weight exponent $B_K=4\gamma_K+2$ compensates for this loss.

\item \textbf{Numerical constant that allows us to use bootstrapping.}
The resulting contraction is bounded by
\[
 c_*=0.8207720540\ldots<1
\]
uniformly in $K$.  The exact logarithmic estimate is displayed in
Section \ref{sec:absorption}.

\item \textbf{Interaction order versus total degree.}
The theorem is first proved for interaction order.  The total-degree
corollary uses the canonical Fourier representatives
$\alpha_j\in\{0,\ldots,K-1\}$ and the elementary inequality
$s(\alpha)\le\sum_j\alpha_j$.
\end{enumerate}


\section{Fourier notation and the constants}

For $f:C_K^N\to\C$, write
\begin{equation}
\label{eq:fourier}
 f(x)=\sum_{\alpha\in\{0,\ldots,K-1\}^N}
 \wh f(\alpha)\chi_\alpha(x),
 \qquad
 \wh f(\alpha)=\E_x f(x)\overline{\chi_\alpha(x)}.
\end{equation}
The characters form an orthonormal basis, so Parseval gives
\begin{equation}
\label{eq:parseval}
 \norm f_2^2=\sum_\alpha\abs{\wh f(\alpha)}^2.
\end{equation}
Decompose $f$ into exact interaction layers:
\begin{equation}
\label{eq:layers}
 f=\sum_{r=0}^N f_r,
 \qquad
 f_r=\sum_{s(\alpha)=r}\wh f(\alpha)\chi_\alpha.
\end{equation}
For $r\ge1$, put
\begin{equation}
\label{eq:pr-Ar}
 p_r=\frac{2r}{r+1},
 \qquad
 A_r(f)=
 \left(\sum_{s(\alpha)=r}
 \abs{\wh f(\alpha)}^{p_r}\right)^{1/p_r}.
\end{equation}

For $d\ge1$, let $\BH^{\mathrm{int}}_{d,K}$ be the least dimension-free
constant such that
\begin{equation}
\label{eq:BH-int-def}
 \left(\sum_{s(\alpha)\le d}
 \abs{\wh f(\alpha)}^{p_d}\right)^{1/p_d}
 \le \BH^{\mathrm{int}}_{d,K}\norm f_\infty
\end{equation}
for all functions supported on interaction orders at most $d$.
Using the canonical representatives, define
\begin{equation}
\label{eq:total-degree}
 \abs\alpha=\alpha_1+\cdots+\alpha_N
\end{equation}
and define $\BH^{\deg}_{d,K}$ analogously for Fourier support in
$\abs\alpha\le d$.  Since $s(\alpha)\le\abs\alpha$,
\begin{equation}
\label{eq:degree-comparison}
 \BH^{\deg}_{d,K}\le\BH^{\mathrm{int}}_{d,K}.
\end{equation}

\section{Analytic inputs}

\subsection{Potts hypercontractivity}

For $0\le\rho\le1$, define
\begin{equation}
\label{eq:potts-noise}
 (T_\rho g)(a)=\rho g(a)+(1-\rho)\E g,
 \qquad a\in C_K.
\end{equation}
On the product, the coordinate operators commute and
\begin{equation}
\label{eq:potts-multiplier}
 T_\rho^{\otimes N}\chi_\alpha
 =\rho^{s(\alpha)}\chi_\alpha.
\end{equation}

\begin{theorem}[Potts hypercontractivity]
\label{thm:potts}
Let $1<p\le2$.  For every $N\ge1$ and every complex-valued
$g:C_K^N\to\C$,
\begin{equation}
\label{eq:potts-hc}
 \norm{T_\rho^{\otimes N}g}_2\le\norm g_p
 \qquad\text{whenever}\qquad
 0\le\rho\le(p-1)^{\gamma_K},
\end{equation}
where $\gamma_K$ is given by \eqref{eq:gamma-intro}.  This is
\cite[Proposition 3.1]{PellegrinoRaposo}.
\end{theorem}

For $u\ge2$, the radius needed later is (with $p_u=\frac{2u}{u+1}$)
\begin{equation}
\label{eq:rho-u}
 \rho_u=(p_u-1)^{\gamma_K}
 =\left(\frac{u-1}{u+1}\right)^{\gamma_K}.
\end{equation}

\subsection{Fixed low interaction levels}

The support-sensitive theorem of
Defant--Galicer--Mansilla--Masty\l o--Muro gives, for fixed $K$, an
exponential-in-$r$ dimension-free bound on interaction order at most $r$
\cite[Theorem 2.13]{DefantEtAlHamming}.  We use only $r=1,2,3$.

\begin{lemma}[Three fixed base levels]
\label{lem:base-levels}
For each fixed $K\ge2$ and each $r\in\{1,2,3\}$, there is a constant
$\beta_{K,r}<\infty$, independent of $N$, such that every exact
interaction-$r$ function $P$ satisfies
\begin{equation}
\label{eq:base-levels}
 A_r(P)\le\beta_{K,r}\norm P_\infty.
\end{equation}
\end{lemma}

\subsection{An interior spectral-projection estimate}

We use  the elementary fact that the  operator $T_\rho$ of Potts hypercontractivity remains Markov on a
certain real interval that contains $0$ strictly inside.

\begin{lemma}[The negative Potts interval]
\label{lem:negative-interval}
For every
\begin{equation}
\label{eq:IK}
 t\in I_K:=\left[-\frac1{K-1},1\right],
\end{equation}
the one-coordinate operator
\[
 T_t=\Pi+t(I-\Pi),
 \qquad \Pi g=(\E g)\mathbf 1,
\]
is a Markov operator (averaging with transition probabilities).  Consequently,
\begin{equation}
\label{eq:negative-Linf}
 \norm{T_t^{\otimes N}f}_\infty\le\norm f_\infty
 \qquad(t\in I_K).
\end{equation}
\end{lemma}

\begin{proof}
The transition kernel is
\begin{equation}
\label{eq:negative-kernel}
 \mathcal K_t(a,b)
 =t\mathbf 1_{\{a=b\}}+\frac{1-t}{K}.
\end{equation}
Its diagonal entries are
$[1+(K-1)t]/K$, its off-diagonal entries are $(1-t)/K$, and every row sums
to one.  These entries are nonnegative exactly for
$-1/(K-1)\le t\le1$.  Tensor products preserve positivity and row sums.
\end{proof}

\medskip

\begin{remark}
The next lemma is very simple but very important. It suggests to split the polynomial on $C_K$ not by the order of homogeneity,
but by the order of the size of the support of monomials. This idea first appeared in \cite{SloteVolbergZhang}, 
but the estimate of the spectral projections on such fixed support size peices of the polynomial was much worse than below. 
This idea is also used in \cite{DefantEtAlHamming}.
\end{remark}

\medskip

\begin{lemma}[Fixed-order interaction projection]
\label{lem:projection}
Fix $K\ge2$ and an integer $r\ge1$.  There is a constant
$L_{K,r}<\infty$ such that, whenever $f:C_K^N\to\C$ has interaction order
at most $M$,
\begin{equation}
\label{eq:projection}
 \norm{f_r}_\infty\le L_{K,r}M^r\norm f_\infty.
\end{equation}
\end{lemma}

\begin{proof}
Fix $x\in C_K^N$ and form the one-variable polynomial
\begin{equation}
\label{eq:Px}
 P_x(t)=T_t^{\otimes N}f(x)
 =\sum_{j=0}^M t^j f_j(x).
\end{equation}
By Lemma \ref{lem:negative-interval},
\begin{equation}
\label{eq:Px-bound}
 \norm{P_x}_{L_\infty(I_K)}\le\norm f_\infty.
\end{equation}
Apply an affine change of variable from $I_K$ to $[-1,1]$.  The point $t=0$
is sent to
\[
 s_K=-\frac{K-2}{K}\in(-1,1).
\]
If $Q_x$ denotes the transformed polynomial and
$G_x(\theta)=Q_x(\cos\theta)$, then $G_x$ is a trigonometric polynomial of
degree at most $M$.  Bernstein's inequality gives
\[
 \norm{G_x^{(j)}}_\infty\le M^j\norm{G_x}_\infty
 \qquad(j\ge1).
\]
Since $s_K$ is a fixed interior point for fixed $K$, repeated differentiation
of $Q_x(s)=G_x(\arccos s)$ at $s=s_K$ gives
\[
 \abs{Q_x^{(r)}(s_K)}
 \le C_{K,r}M^r\norm{Q_x}_{L_\infty[-1,1]}.
\]
Undoing the affine change of variable and using
$P_x^{(r)}(0)=r!f_r(x)$ proves \eqref{eq:projection}.
\end{proof}

\begin{corollary}[Low-level closure]
\label{cor:low}
For every fixed $K\ge2$ and every $B>0$, there is a constant
$C_{K,B}<\infty$ such that every $f$ of interaction order at most $M$
satisfies
\begin{equation}
\label{eq:low}
 \left[
  \sum_{r=1}^{\min\{3,M\}}
  \frac{A_r(f)^2}{r^{2B}}
 \right]^{1/2}
 \le C_{K,B}M^3\norm f_\infty.
\end{equation}
\end{corollary}

\begin{proof}
For $r=1,2,3$, Lemmas \ref{lem:base-levels} and \ref{lem:projection} give
\[
 A_r(f)=A_r(f_r)
 \le\beta_{K,r}\norm{f_r}_\infty
 \le\beta_{K,r}L_{K,r}M^r\norm f_\infty.
\]
There are only three terms, and $M^r\le M^3$.
\end{proof}

\section{Mixed-norm tools}

The next two lemmas are numerical.  They do not depend on the underlying
group.

\begin{lemma}[Mixed-angle summation]
\label{lem:mixed-angle}
Let $0<a<1/2$.  Suppose $u_i,v_i\ge0$ satisfy
$\sum_i u_i\le1$ and $\sum_i v_i\le1$, and let
$\theta_i\in[a,1-a]$.  Then
\begin{equation}
\label{eq:mixed-angle}
 \sum_i u_i^{\theta_i}v_i^{1-\theta_i}
 \le 2a^a(1-a)^{1-a}
 =2\exp\{-h(a)\},
\end{equation}
where, as in \cite[Lemma 5.2]{PellegrinoTeixeira},
\begin{equation}
\label{eq:entropy}
 h(t)=-t\log t-(1-t)\log(1-t).
\end{equation}
\end{lemma}

\begin{proof}
Split the indices into $I_+=\{i:u_i\ge v_i\}$ and
$I_-=\{i:u_i<v_i\}$.  On $I_+$,
$u_i^{\theta_i}v_i^{1-\theta_i}\le u_i^{1-a}v_i^a$, while on $I_-$,
$u_i^{\theta_i}v_i^{1-\theta_i}\le u_i^av_i^{1-a}$.  H\"older's inequality
gives
\[
 \sum_i u_i^{\theta_i}v_i^{1-\theta_i}
 \le U_+^{1-a}V_+^a+U_-^aV_-^{1-a},
\]
where $U_\pm,V_\pm$ are the corresponding partial masses.  Enlarging unused
mass, assume $U_++U_-=V_++V_-=1$.  Write $U_+=x$ and $V_+=y$, with
$0\le y\le x\le1$.  A second application of H\"older yields
\[
 x^{1-a}y^a+(1-x)^a(1-y)^{1-a}
 \le(1+x-y)^{1-a}(1-x+y)^a.
\]
With $\delta=x-y\in[0,1]$, the last expression is
$(1+\delta)^{1-a}(1-\delta)^a$.  Its maximum occurs at
$\delta=1-2a$ and equals $2a^a(1-a)^{1-a}$.
\end{proof}

For a finite array $c=(c_{\alpha,\beta})$, define
\begin{align}
 X_{u,v}(c)
 &=\left[
   \sum_\alpha
   \left(\sum_\beta\abs{c_{\alpha,\beta}}^2\right)^{p_u/2}
  \right]^{1/p_u},
 \label{eq:X-def}\\
 Y_{u,v}(c)
 &=\left[
   \sum_\beta
   \left(\sum_\alpha\abs{c_{\alpha,\beta}}^2\right)^{p_v/2}
  \right]^{1/p_v}.
 \label{eq:Y-def}
\end{align}

\begin{lemma}[Two-block coefficient interpolation]
\label{lem:two-block}
Let $u,v\ge1$, set $r=u+v$, and let $c=(c_{\alpha,\beta})$ be a finite
scalar array.  Then
\begin{equation}
\label{eq:two-block}
 \left(\sum_{\alpha,\beta}
       \abs{c_{\alpha,\beta}}^{p_r}\right)^{1/p_r}
 \le X_{u,v}(c)^{u/r}Y_{u,v}(c)^{v/r}.
\end{equation}
\end{lemma}

\begin{proof}
Interpolate finite mixed norms between
$\ell_{p_u}(\alpha;\ell_2(\beta))$ and
$\ell_2(\alpha;\ell_{p_v}(\beta))$ with parameter $v/r$.  The two
interpolated exponents are both $p_r$, because
\[
 \frac1{p_r}
 =\frac ur\frac1{p_u}+\frac vr\frac12
 =\frac ur\frac12+\frac vr\frac1{p_v}.
\]
Thus
\[
 \norm c_{\ell_{p_r}(\alpha;\ell_{p_r}(\beta))}
 \le X_{u,v}(c)^{u/r}
      \norm c_{\ell_2(\alpha;\ell_{p_v}(\beta))}^{v/r}.
\]
Since $p_v\le2$, Minkowski's inequality gives
\[
 \norm c_{\ell_2(\alpha;\ell_{p_v}(\beta))}
 \le\norm c_{\ell_{p_v}(\beta;\ell_2(\alpha))}
 =Y_{u,v}(c).
\]
The left-hand side is the ordinary $\ell_{p_r}$ norm of the full array.
\end{proof}

\section{The weighted row--column bootstrap}

Fix $B>0$ and integers $M,N\ge1$.  Define
\begin{equation}
\label{eq:WB}
 \mathcal W_B(f)=
 \left[\sum_{r=1}^M\frac{A_r(f)^2}{r^{2B}}\right]^{1/2}
\end{equation}
and
\begin{equation}
\label{eq:Gamma}
 \Gamma^{(K)}_{M,N}(B)
 =\sup\left\{
 \frac{\mathcal W_B(f)}{\norm f_\infty}:
 1\le n\le N,\quad 0\ne f:C_K^n\to\C,\quad
 \wh f(\alpha)=0\text{ if }s(\alpha)>M
 \right\}.
\end{equation}
For fixed $M,N$, this is finite by finite dimensionality.  The cutoff $N$ is
retained until the contraction has been absorbed.

\begin{remark}
The definition of $ \mathcal W_B(F)$ seems overcomplicated on the first glance. But in fact, its form has a simple explanation.
This grading in  homogeneity $r$ and keeping all $r\in [1, M]$ together is enevitable by the following reason.  
As we cannot claim that $Q_{\omega, d, e}$ obtained from bounded $F$ has a nice $L^\infty$  norm for fixed $d, e$ (see Addendum A for this claim), 
one is forced to consider the whole $Q_{\omega}$ as the source of nice norm, but $Q_\omega $
has of course many homogeneity modes; 
so one should invent a method to estimate all modes simultaneously—which explains the square function $W_B(F)$ in formula \eqref{eq:WB}
 as a vehicle for 
subsequent estimates. 
\end{remark}

\begin{remark}
We cannot exclude that by modifying the square function $\mathcal W_B$ one can improve further the estimate.
\end{remark}

Let
\begin{equation}
\label{eq:Q-bidecomp}
 Q(x,y)=\sum_{u,v\ge0}Q_{u,v}(x,y)
\end{equation}
be a function on $C_K^{n_x}\times C_K^{n_y}$, with $n_x,n_y\le N$, whose
Fourier support satisfies $u+v\le M$.  Here $Q_{u,v}$ has interaction order
$u$ in $x$ and $v$ in $y$:
\begin{equation}
\label{eq:Quv}
 Q_{u,v}(x,y)
 =\sum_{\substack{s(\alpha)=u\\s(\beta)=v}}
 c^{u,v}_{\alpha,\beta}\chi_\alpha(x)\chi_\beta(y).
\end{equation}
For this coefficient array, abbreviate $X_{u,v}=X_{u,v}(c^{u,v})$ and
$Y_{u,v}=Y_{u,v}(c^{u,v})$.  When $u\ge2$, define
\begin{equation}
\label{eq:Xtilde}
 \widetilde X_{u,v}=\rho_u^vX_{u,v},
\end{equation}
and, when $v\ge2$, define
\begin{equation}
\label{eq:Ytilde}
 \widetilde Y_{u,v}=\rho_v^uY_{u,v},
\end{equation}
with $\rho_u$ from \eqref{eq:rho-u}.

\begin{lemma}[Simultaneous Potts hypercontractivity row and column control]
\label{lem:row-column}
For every $B>0$,
\begin{align}
 \sum_{\substack{u\ge2\\v\ge0}}
 \frac{\widetilde X_{u,v}^2}{u^{2B}}
 &\le
 \Gamma^{(K)}_{M,N}(B)^2\norm Q_\infty^2,
 \label{eq:row-control}\\
 \sum_{\substack{u\ge0\\v\ge2}}
 \frac{\widetilde Y_{u,v}^2}{v^{2B}}
 &\le
 \Gamma^{(K)}_{M,N}(B)^2\norm Q_\infty^2.
 \label{eq:column-control}
\end{align}
\end{lemma}

\begin{proof}
By symmetry, prove \eqref{eq:row-control}.  Fix $u\ge2$.  We need this condition because for $p_u=\frac{2u}{u+1}$ we cannot have $p_u=1$, now hypercontarctivity $(2,1)$ is possible.
For
$s(\alpha)=u$, put
\begin{equation}
\label{eq:Bualpha}
 B_{u,\alpha}(y)
 =\sum_{v\ge0}\sum_{s(\beta)=v}
 c^{u,v}_{\alpha,\beta}\chi_\beta(y).
\end{equation}
Because $p_u\le2$, Minkowski's inequality, followed by Parseval in the
$y$ variables, gives
\begin{align*}
 \left(\sum_{v\ge0}\widetilde X_{u,v}^2\right)^{1/2}
 &\le
 \left[
  \sum_{s(\alpha)=u}
  \left(
   \sum_{v\ge0}\rho_u^{2v}
   \sum_{s(\beta)=v}\abs{c^{u,v}_{\alpha,\beta}}^2
  \right)^{p_u/2}
 \right]^{1/p_u}\\
 &=\left[
  \sum_{s(\alpha)=u}
  \norm{T_{\rho_u}^{\otimes n_y}B_{u,\alpha}}_2^{p_u}
 \right]^{1/p_u}.
\end{align*}
Theorem \ref{thm:potts} gives
\[
 \norm{T_{\rho_u}^{\otimes n_y}B_{u,\alpha}}_2
 \le\norm{B_{u,\alpha}}_{p_u}.
\]
Set
\[
 g_u(y)=
 \left[
  \sum_{s(\alpha)=u}
  \abs{B_{u,\alpha}(y)}^{p_u}
 \right]^{1/p_u}.
\]
Since the measure is normalized and $p_u\le2$,
\begin{equation}
\label{eq:fixed-u}
 \left(\sum_{v\ge0}\widetilde X_{u,v}^2\right)^{1/2}
 \le\norm{g_u}_{p_u}\le\norm{g_u}_2.
\end{equation}
Multiplying by $u^{-B}$, squaring, summing in $u$, and applying Tonelli yields
\begin{align*}
 \sum_{\substack{u\ge2\\v\ge0}}
 \frac{\widetilde X_{u,v}^2}{u^{2B}}
 &\le
 \E_y\sum_{u=2}^M\frac{g_u(y)^2}{u^{2B}}.
\end{align*}
For fixed $y$, $g_u(y)$ is precisely the coefficient $\ell_{p_u}$ norm of
the interaction-$u$ layer of $x\mapsto Q(x,y)$.  Therefore the definition
of $\Gamma^{(K)}_{M,N}(B)$ bounds the last display by
\[
 \E_y\left[
  \Gamma^{(K)}_{M,N}(B)^2
  \norm{Q(\,\cdot\,,y)}_\infty^2
 \right]
 \le\Gamma^{(K)}_{M,N}(B)^2\norm Q_\infty^2.
\]
The column estimate follows by interchanging $x$ and $y$.
\end{proof}

\section{Central entropy contraction}

For $r\ge4$, define the central bidegree window
\begin{equation}
\label{eq:Cr}
 \mathcal C_r=
 \left\{(u,v)\in\N^2:
 u+v=r,\quad u,v\ge2,\quad
 \frac r4\le u,v\le\frac{3r}{4}
 \right\}.
\end{equation}
The explicit conditions $u,v\ge2$ are essential at the bottom of the
bootstrap.  Put
\begin{equation}
\label{eq:kappa4}
 \kappa_4=\sup_{r\ge4}(r+1)^{1/(2r)}=5^{1/8}.
\end{equation}

\begin{lemma}[Uniform cost of removing the Potts weights]
\label{lem:potts-cost}
Let $u,v\ge2$, $r=u+v$, and $\theta=u/r$.  Then
\begin{equation}
\label{eq:potts-cost}
 X_{u,v}^{\theta}Y_{u,v}^{1-\theta}
 \le3^{2\gamma_K}
 \widetilde X_{u,v}^{\theta}
 \widetilde Y_{u,v}^{1-\theta}.
\end{equation}
\end{lemma}

\begin{proof}
By definition, the ratio between the two sides without the constant is
\[
 R_K(u,v)
 =\rho_u^{-v\theta}\rho_v^{-u(1-\theta)}.
\]
Using \eqref{eq:rho-u},
\begin{equation}
\label{eq:log-R}
 \log R_K(u,v)
 =\gamma_K\frac{uv}{r}
 \left[
  \log\frac{u+1}{u-1}
  +\log\frac{v+1}{v-1}
 \right].
\end{equation}
For $t\ge2$, define
\[
 \phi(t)=t\log\frac{t+1}{t-1}.
\]
The function $\phi$ is decreasing.  Indeed, with $z=1/t\in(0,1/2]$,
\begin{align*}
 \log\frac{t+1}{t-1}
 &=2\sum_{j\ge0}\frac{z^{2j+1}}{2j+1},\\
 \frac{2t}{t^2-1}
 &=2\sum_{j\ge0}z^{2j+1},
\end{align*}
so
$\phi'(t)=\log((t+1)/(t-1))-2t/(t^2-1)\le0$.
Therefore $\phi(t)\le\phi(2)=2\log3$.  Rewriting
\eqref{eq:log-R} as
\[
 \log R_K(u,v)
 =\frac{\gamma_K}{r}\bigl[v\phi(u)+u\phi(v)\bigr]
\]
gives
\[
 \log R_K(u,v)\le2\gamma_K\log3.
\]
Exponentiation proves the claim.
\end{proof}

\begin{proposition}[Central Potts entropy contraction]
\label{prop:central}
For every $B>0$, every $M\ge4$, and every $Q$ as in
\eqref{eq:Q-bidecomp},
\begin{align}
&\left[
 \sum_{r=4}^M\frac1{r^{2B}}
 \left(
  \sum_{(u,v)\in\mathcal C_r}
  \norm{c(Q_{u,v})}_{p_r}^{p_r}
 \right)^{2/p_r}
\right]^{1/2}
\notag\\
&\qquad\le
3^{2\gamma_K}\kappa_4\sqrt2\,
\exp\left\{-\left(B+\frac12\right)
 h\left(\frac14\right)\right\}
\Gamma^{(K)}_{M,N}(B)\norm Q_\infty.
\label{eq:central}
\end{align}
Here $c(Q_{u,v})=(c^{u,v}_{\alpha,\beta})$.
\end{proposition}

\begin{proof}
Fix $(u,v)\in\mathcal C_r$ and write $\theta=u/r$.  Set
\[
 U_{u,v}=\frac{\widetilde X_{u,v}}{u^B},
 \qquad
 V_{u,v}=\frac{\widetilde Y_{u,v}}{v^B}.
\]
Lemmas \ref{lem:two-block} and \ref{lem:potts-cost} give
\begin{align*}
 \frac{\norm{c(Q_{u,v})}_{p_r}}{r^B}
 &\le3^{2\gamma_K}
 \frac{u^{B\theta}v^{B(1-\theta)}}{r^B}
 U_{u,v}^{\theta}V_{u,v}^{1-\theta}\\
 &=3^{2\gamma_K}\exp\{-Bh(\theta)\}
 U_{u,v}^{\theta}V_{u,v}^{1-\theta}\\
 &\le3^{2\gamma_K}
 \exp\left\{-Bh\left(\frac14\right)\right\}
 U_{u,v}^{\theta}V_{u,v}^{1-\theta}.
\end{align*}
The final step uses $\theta\in[1/4,3/4]$ and the symmetry and concavity of
binary entropy.

By Lemma \ref{lem:row-column}, both square sums
$\sum U_{u,v}^2$ and $\sum V_{u,v}^2$, over all central pairs and all $r$,
are at most
\[
 G^2=\Gamma^{(K)}_{M,N}(B)^2\norm Q_\infty^2.
\]
If $G=0$, there is nothing to prove.  Otherwise apply Lemma
\ref{lem:mixed-angle} with $a=1/4$ to
\[
 \frac{U_{u,v}^2}{G^2},
 \qquad
 \frac{V_{u,v}^2}{G^2},
 \qquad
 \theta_{u,v}=\frac ur.
\]
It follows that
\begin{align}
&\sum_{r=4}^M\sum_{(u,v)\in\mathcal C_r}
 \left(\frac{\norm{c(Q_{u,v})}_{p_r}}{r^B}\right)^2
\notag\\
&\qquad\le
2\,3^{4\gamma_K}
\exp\left\{-(2B+1)h\left(\frac14\right)\right\}G^2.
\label{eq:central-l2}
\end{align}
Finally,
$1/p_r-1/2=1/(2r)$ and $\#\mathcal C_r\le r+1$, so
\begin{align*}
 \left(
  \sum_{(u,v)\in\mathcal C_r}
  \norm{c(Q_{u,v})}_{p_r}^{p_r}
 \right)^{1/p_r}
 &\le(r+1)^{1/(2r)}
 \left(
  \sum_{(u,v)\in\mathcal C_r}
  \norm{c(Q_{u,v})}_{p_r}^{2}
 \right)^{1/2}\\
 &\le\kappa_4
 \left(
  \sum_{(u,v)\in\mathcal C_r}
  \norm{c(Q_{u,v})}_{p_r}^{2}
 \right)^{1/2}.
\end{align*}
Square, sum in $r$, use \eqref{eq:central-l2}, and take square roots.
\end{proof}

\section{Random coordinate splitting}

Let $f:C_K^n\to\C$ have interaction order at most $M$, with $n\le N$.
Color each coordinate independently $x$ or $y$, with probability $1/2$ for
each color.  For a coloring $\omega$, define
\begin{equation}
\label{eq:Qomega}
 Q_\omega(x,y)=f(z_1,\ldots,z_n),
 \qquad
 z_j=
 \begin{cases}
 x_j,&j\text{ is colored }x,\\
 y_j,&j\text{ is colored }y.
 \end{cases}
\end{equation}
Then
\begin{equation}
\label{eq:norm-equality}
 \norm{Q_\omega}_\infty=\norm f_\infty.
\end{equation}
Indeed, the relevant coordinate in each pair $(x_j,y_j)$ may be chosen
arbitrarily, so every point of $C_K^n$ is represented.

For $r\ge4$, define the central coefficient mass
\begin{equation}
\label{eq:Cr-omega}
 C_r(\omega)=
 \left(
  \sum_{(u,v)\in\mathcal C_r}
  \norm{c(Q_{\omega,u,v})}_{p_r}^{p_r}
 \right)^{1/p_r}.
\end{equation}

\begin{lemma}[Exact ancestry and coefficient capture]
\label{lem:capture}
For every $r\ge4$,
\begin{equation}
\label{eq:capture-exact}
 \E_\omega C_r(\omega)^{p_r}
 =\pi_r A_r(f)^{p_r},
\end{equation}
where
\begin{equation}
\label{eq:pi-r}
 \pi_r=
 \Pr\{(Z_r,r-Z_r)\in\mathcal C_r\},
 \qquad Z_r\sim\Bin(r,1/2).
\end{equation}
Moreover,
\begin{equation}
\label{eq:pi-lower}
 \pi_r\ge\frac13\qquad(r\ge4).
\end{equation}
Consequently,
\begin{equation}
\label{eq:Ar-capture}
 A_r(f)
 \le3^{1/p_r}
 \left(\E_\omega C_r(\omega)^2\right)^{1/2}
 \le3^{5/8}
 \left(\E_\omega C_r(\omega)^2\right)^{1/2}.
\end{equation}
\end{lemma}

\begin{proof}
Let $I_x$ and $I_y$ be the color classes.  A character $\chi_\alpha$ becomes
\begin{equation}
\label{eq:split-character}
 \chi_{\alpha|_{I_x}}(x)\chi_{\alpha|_{I_y}}(y).
\end{equation}
The restrictions retain both the coordinate and its full local degree in
$\Z/K\Z$.  Their supports are disjoint and their union reconstructs $\alpha$.
Thus the parent-to-child map is injective and no coefficients merge.

If $s(\alpha)=r$, its $x$-interaction order is the number of active
coordinates colored $x$, hence is distributed as $Z_r\sim\Bin(r,1/2)$.
Summing coefficient indicators proves \eqref{eq:capture-exact}.

For $r=4$,
\[
 \pi_4=\Pr\{Z_4=2\}=\frac38,
\]
and for $r=5$,
\[
 \pi_5=\Pr\{Z_5\in\{2,3\}\}=\frac58.
\]
For $r\ge6$, Chebyshev's inequality gives
\[
 \pi_r
 \ge1-\Pr\left\{\abs{Z_r-r/2}\ge\frac r4\right\}
 \ge1-\frac4r
 \ge\frac13.
\]
This proves \eqref{eq:pi-lower}.  Since $p_r\le2$ and the coloring space has
probability one,
\[
 \pi_r^{1/p_r}A_r(f)
 =\norm{C_r}_{p_r}
 \le\norm{C_r}_2.
\]
Finally, $1/p_r=1/2+1/(2r)\le5/8$ for $r\ge4$.
\end{proof}

\section{Uniform contraction and bootstrapping}
\label{sec:absorption}

Combining the preceding two sections gives a high-level contraction for any
weight exponent $B$.

\begin{proposition}[High interaction levels]
\label{prop:high-levels}
For every $B>0$ and every $f$ of interaction order at most $M$,
\begin{equation}
\label{eq:high-levels}
 \left[
  \sum_{r=4}^M\frac{A_r(f)^2}{r^{2B}}
 \right]^{1/2}
 \le c_{K,B}\Gamma^{(K)}_{M,N}(B)\norm f_\infty,
\end{equation}
where
\begin{equation}
\label{eq:cKB}
 c_{K,B}
 =3^{2\gamma_K+5/8}5^{1/8}\sqrt2\,
 \exp\left\{-\left(B+\frac12\right)
 h\left(\frac14\right)\right\}.
\end{equation}
\end{proposition}

\begin{proof}
By Lemma \ref{lem:capture}, Tonelli's theorem, Proposition
\ref{prop:central}, and \eqref{eq:norm-equality},
\begin{align*}
 \sum_{r=4}^M\frac{A_r(f)^2}{r^{2B}}
 &\le3^{5/4}\E_\omega
 \sum_{r=4}^M\frac{C_r(\omega)^2}{r^{2B}}\\
 &\le3^{5/4}
 \left[
  3^{2\gamma_K}5^{1/8}\sqrt2\,
  \exp\left\{-\left(B+\frac12\right)
  h\left(\frac14\right)\right\}
 \right]^2\\
 &\qquad\qquad\times
 \Gamma^{(K)}_{M,N}(B)^2\norm f_\infty^2.
\end{align*}
Taking square roots gives \eqref{eq:high-levels}.
\end{proof}

Choose
\begin{equation}
\label{eq:BK}
 B_K=4\gamma_K+2.
\end{equation}
Let
\begin{equation}
\label{eq:h0}
 h_0=h\left(\frac14\right)
 =2\log2-\frac34\log3.
\end{equation}

\begin{lemma}[Uniform numerical contraction]
\label{lem:numerical}
For every $K\ge2$,
\begin{equation}
\label{eq:c-uniform}
 c_{K,B_K}\le c_*,
\end{equation}
where
\begin{equation}
\label{eq:cstar}
 c_*
 =3^{13/8}5^{1/8}\sqrt2\,\e^{-(9/2)h_0}
 =0.8207720540\ldots<1.
\end{equation}
\end{lemma}

\begin{proof}
First, $\gamma_K\ge1/2$.  This is immediate for $K=2$.  For $K\ge3$, put
$t=K-1\ge2$.  The inequality is equivalent to
\[
 \log t\ge\frac{2(t-1)}{t+1}.
\]
The difference vanishes at $t=1$ and has derivative
\[
 \frac1t-\frac4{(t+1)^2}
 =\frac{(t-1)^2}{t(t+1)^2}\ge0.
\]
Using \eqref{eq:cKB} and \eqref{eq:BK},
\begin{align*}
 \frac{c_{K,B_K}}{c_*}
 &=3^{2(\gamma_K-1/2)}
   \exp\{-4(\gamma_K-1/2)h_0\}\\
 &=\left(\frac{243}{256}\right)^{\gamma_K-1/2}
 \le1,
\end{align*}
because
$2\log3-4h_0=5\log3-8\log2=\log(243/256)<0$.
Finally,
\begin{equation}
\label{eq:log-cstar}
 \log c_*
 =5\log3+\frac18\log5-\frac{17}{2}\log2
 =-0.1975098523\ldots<0.
\end{equation}
This proves $c_*<1$ and the stated numerical value.
\end{proof}

\begin{theorem}[Uniform weighted estimate]
\label{thm:weighted}
For every fixed $K\ge2$, there is a constant $D_K<\infty$ such that, for
all $M,N\ge1$,
\begin{equation}
\label{eq:Gamma-bound}
 \Gamma^{(K)}_{M,N}(B_K)\le D_KM^3.
\end{equation}
Consequently, in every dimension,
\begin{equation}
\label{eq:weighted-final}
 \left[
  \sum_{r=1}^M
  \frac{A_r(f)^2}{r^{2B_K}}
 \right]^{1/2}
 \le D_KM^3\norm f_\infty.
\end{equation}
\end{theorem}

\begin{proof}
For $f$ entering the supremum in \eqref{eq:Gamma}, split the weighted square
function at level $3$.  Corollary \ref{cor:low}, Proposition
\ref{prop:high-levels}, and Lemma \ref{lem:numerical} give
\begin{align*}
 \mathcal W_{B_K}(f)
 &\le
 \left[
  \sum_{r=1}^{\min\{3,M\}}
  \frac{A_r(f)^2}{r^{2B_K}}
 \right]^{1/2}
 +
 \left[
  \sum_{r=4}^M
  \frac{A_r(f)^2}{r^{2B_K}}
 \right]^{1/2}\\
 &\le C_KM^3\norm f_\infty
 +c_*\Gamma^{(K)}_{M,N}(B_K)\norm f_\infty.
\end{align*}
Taking the supremum yields
\[
 \Gamma^{(K)}_{M,N}(B_K)
 \le C_KM^3+c_*\Gamma^{(K)}_{M,N}(B_K).
\]
Since $c_*<1$,
\[
 \Gamma^{(K)}_{M,N}(B_K)
 \le\frac{C_K}{1-c_*}M^3.
\]
The right-hand side is independent of $N$, so one may let $N\to\infty$.
\end{proof}

\section{Polynomial cyclic Bohnenblust--Hille bounds}

\begin{theorem}[Interaction-order inequality]
\label{thm:interaction}
For each fixed $K\ge2$, there is a constant $C_K<\infty$ such that
\begin{equation}
\label{eq:interaction-final}
 \BH^{\mathrm{int}}_{d,K}
 \le C_Kd^{4\gamma_K+5}
 \qquad(d\ge1).
\end{equation}
\end{theorem}

\begin{proof}
Let $f$ have interaction order at most $d$.  For each $r\le d$, since
$p_r\le p_d$, monotonicity of finite $\ell_p$ norms gives
\[
 \left(
  \sum_{s(\alpha)=r}\abs{\wh f(\alpha)}^{p_d}
 \right)^{1/p_d}
 \le A_r(f).
\]
Put $b_r=A_r(f)/r^{B_K}$.  Then Theorem \ref{thm:weighted}, with $M=d$,
gives
\begin{align*}
 \left(
  \sum_{\substack{\alpha\\1\le s(\alpha)\le d}}
  \abs{\wh f(\alpha)}^{p_d}
 \right)^{1/p_d}
 &\le\left(\sum_{r=1}^d A_r(f)^{p_d}\right)^{1/p_d}\\
 &\le d^{B_K}
 \left(\sum_{r=1}^d b_r^{p_d}\right)^{1/p_d}\\
 &\le d^{B_K}d^{1/p_d-1/2}
 \left(\sum_{r=1}^d b_r^2\right)^{1/2}\\
 &\le D_Kd^{B_K+3+1/(2d)}\norm f_\infty.
\end{align*}
The factor $d^{1/(2d)}$ is uniformly bounded.  Also
$\abs{\wh f(0)}\le\norm f_\infty$.  Since $B_K+3=4\gamma_K+5$, the result
follows after enlarging the constant.
\end{proof}

\begin{corollary}[Exact interaction layer]
\label{cor:exact}
If $f$ is supported on $s(\alpha)=d$, then
\begin{equation}
\label{eq:exact-final}
 A_d(f)\le C_Kd^{4\gamma_K+5}\norm f_\infty.
\end{equation}
\end{corollary}

\begin{proof}
The single $r=d$ term in \eqref{eq:weighted-final} gives
$A_d(f)/d^{B_K}\le D_Kd^3\norm f_\infty$.
\end{proof}

\begin{corollary}[Canonical total degree]
\label{cor:degree}
For every fixed $K\ge2$,
\begin{equation}
\label{eq:degree-final}
 \BH^{\deg}_{d,K}
 \le C_Kd^{4\gamma_K+5}.
\end{equation}
In particular, for $K\ge3$,
\begin{equation}
\label{eq:degree-power}
 \BH^{\deg}_{d,K}
 \le C_Kd^{5+K\log(K-1)/(K-2)}.
\end{equation}
\end{corollary}

\begin{proof}
Use \eqref{eq:degree-comparison} and Theorem \ref{thm:interaction}.
\end{proof}

\section{What is the difference with \cite{PellegrinoRaposo}}

It is worth emphasizing the distinction between this argument and the
subexponential recurrence of \cite{PellegrinoRaposo}.  The latter proceeds by
homogenizing a Fourier character into $d$ positions, passing to a symmetric
$d$-linear form, splitting the positions into blocks, 
and polarizing a mixed evaluation back to the diagonal.  Its
recurrence contains both a Potts loss and a binomial orbit--polarization loss.

None of those steps appears here.  The weighted method works directly with
the canonical Fourier coefficients and all interaction layers at once.  A
coordinate is colored as a whole, so the local degree of each variable is transported
without alteration.  The entropy gain from balanced support splits is paid
once inside a square function and then absorbed.  In that precise sense, the
only part of the cyclic recurrence needed at high degree is its Potts
hypercontractive estimate.

\begin{remark}[Why modular arithmetic is harmless here]
The obstruction in cyclic polynomial extensions is often that products of
local frequencies add modulo $K$.  The present splitting never multiplies or
combines two frequencies on the same coordinate.  It only records whether a
coordinate-frequency pair belongs to the $x$ block or the $y$ block.  Exact
ancestry therefore survives unchanged.
\end{remark}

\begin{remark}[Why support order is the right parameter]
At the level of total algebraic degree, a character may contain a large local
power $x_j^{K-1}$.  At the level of interaction order, that coordinate still
contributes exactly one unit.  Thus every interaction-order character is
square-free in the geometry relevant to random coloring.  This is the cyclic
analogue of the absence of dominant powers on the Boolean cube.
\end{remark}

\begin{remark}[The Boolean endpoint]
For $K=2$, $\gamma_2=1/2$ and the argument gives the power $d^7$.  This is the
same weighted Boolean mechanism, with the central bootstrap started at level
$4$ rather than level $6$ as it has been done in \cite{SloteVolberg}.
\end{remark}

\begin{remark}[Scope]
The estimate is polynomial in $d$ for each fixed $K$.  It is not uniform in
$K$: the exponent behaves as $5+\log K+o(1)$, and the implicit constant also
depends on $K$ through the three low-level estimates and the interior
projection constants.
\end{remark}

\appendix

\section{Numerics}

The worst contraction occurs at the smallest allowed value
$\gamma_K=1/2$.  Its exact logarithm is
\[
 \log c_*
 =5\log3+\frac18\log5-\frac{17}{2}\log2.
\]
Direct interval evaluation gives
\[
 -0.197511<\log c_*<-0.197509,
 \qquad
 0.820771<c_*<0.820773.
\]
Thus the absorption margin is larger than $0.179$.

For reference, the powers furnished by the theorem begin as follows:
\[
\begin{array}{c|c|c}
K&\gamma_K&4\gamma_K+5\\
\hline
2&0.500000&7.000000\\
3&0.519860&7.079442\\
4&0.549306&7.197225\\
5&0.577623&7.310491\\
6&0.603539&7.414157
\end{array}
\]

\section{
Declaration of generative AI and AI-assisted technologies in the manuscript
preparation process}
During the preparation of this work, the authors used ChatGPT Pro (OpenAI) to assist with exploratory
proof development, especially in analyzing Section 5 of \cite{PellegrinoTeixeira}. The authors reviewed and verified the mathematical content, edited the
manuscript, and take full responsibility for its content.

\end{document}